\documentclass[a4paper,reqno, 10pt]{amsart}
\usepackage{amsmath,amssymb,amsthm,enumerate}

\usepackage{graphics}
\usepackage{epsfig}
\usepackage{amsfonts}
\usepackage{amssymb}
\usepackage{color}
\usepackage{tikz}
\usetikzlibrary{arrows, calc}
\usetikzlibrary{calc,quotes,angles}
\usepackage{pgfplots}
\usepackage{mathtools}
\usetikzlibrary{intersections}

\usepackage{mathrsfs}

\usepackage{refcheck}
\usepackage{mathtools}

\def\ifl{\iffalse }

\def\bc{\begin{center}} \def\ec{\end{center}}
\def\ba{\begin{array}} \def\ea{\end{array}}
\def\bea{\begin{eqnarray}} \def\eea{\end{eqnarray}}
\def\beaa{\begin{eqnarray*}} \def\eeaa{\end{eqnarray*}}

\theoremstyle{definition}
\newtheorem{thm}{Theorem}[section]
\newtheorem{prop}{Proposition}[section]
\newtheorem{lem}{Lemma}[section]

\theoremstyle{remark}
\newtheorem{rem}{Remark}[section]
\newtheorem*{rem*}{Remark}

\numberwithin{equation}{section}

\newcommand{\R}{\mathbb{R}}

\newcommand{\supp}{\mathop{\mathrm{supp}}}

\newcommand{\pa}{\partial}
\newcommand{\na}{\nabla}
\newcommand{\al}{\alpha}

\newcommand{\Ga}{\Gamma}

\newcommand{\Lg}{\langle}
\newcommand{\Rg}{\rangle}

\newcommand{\OK}{\operatorname{OK}}

\title[GWP and scattering of 3D KGZ]{on the global well-posedness and scattering of the 3D Klein-Gordon-Zakharov system}

 \author[X. Cheng]{Xinyu Cheng}
 \address{X. Cheng, School of Mathematical Sciences, Fudan University, Shanghai, P.R. China}
\email{xycheng@fudan.edu.cn}

 \author[J. Xu]{Jiao Xu}
 \address{J. Xu, School of Mathematics, South China University of Technology,
  Guangzhou, P.R. China}
\email{mathxujiao@scut.edu.cn}
\begin{document}

\maketitle
\begin{abstract}
    In this paper we are interested in the global well-posedness of the 3D Klein-Gordon-Zakharov equations with small non-compactly supported initial data. We show the uniform boundedness of the energy for the global solution without any compactness assumptions on the initial data. The main novelty of our proof is to apply a modified Alinhac's ghost weight method together with a newly developed normal-form type estimate to remedy the lack of the space-time scaling vector field; moreover, we give a clear description of the smallness conditions on the initial data.
\end{abstract}

\section{introduction}
Throughout this work we consider the following Klein-Gordon-Zakharov system (KGZ) in $\R^+\times\R^3$:
\begin{equation}\label{eq:kg}\tag{KGZ}
    \begin{cases}
    \Box \phi+\phi=-n\phi, \\
    \Box n=\Delta |\phi|^2,\\
    (\phi,\partial_t \phi, n,\pa_{t}n)|_{t=0}=(\phi_0, \phi_1, n_0, n_1).
    \end{cases}
\end{equation}
Here $\square = \partial_{tt} - \Delta$ is the d'Alembertian and $\Delta$ is the usual Laplacian. The unknowns $\phi,n$ take values in $\R^3$ and $\R$ respectively; in fact they can be viewed as the Klein-Gordon and wave components in the coupled \eqref{eq:kg} system. The system above plays an important role in plasma physics
where it describes the interaction between Langmuir waves and ion sound
waves in plasma via ion density fluctuation $n$ and the electric field $\phi$ (cf. \cite{Plasma}).

In this paper, we are interested in the small data global well-posedness (GWP) of the model problem \eqref{eq:kg} and the long-term energy/pointwise asymptotic behavior of the global solutions. Moreover, we carefully explore the scattering phenomenon of \eqref{eq:kg}. Our main result is stated below:

 \begin{thm}\label{thm}
Consider the Klein-Gordon-Zakharov system as in \eqref{eq:kg} and let $K$ {be an integer no less than 8}. We denote the energy space $\mathcal{X}_K=H^{K+1}\times H^{K}\times H^K\times H^{K-1}$, where $H^m$ are standard Sobolev spaces in $\R^3$. There exists small $\varepsilon_0>0$ such that if $\varepsilon\in (0,\varepsilon_0)$ and all initial data $(\phi_0,\phi_1,n_0,n_1)$ satisfying the smallness condition below: 
{\begin{equation}\label{ID2}
\begin{split}
    &\|\Lg x \Rg^{K}\Lg \na\Rg^{K+1}\phi_0\|_2+\|\Lg x \Rg^{K}\Lg \na\Rg^{K}\phi_1\|_2+\|\Lg x\Rg^{k+6}\Lg \na \Rg^{k+4}\phi_0\|_2\\&+ \|\Lg x\Rg^{k+6}\Lg \na \Rg^{k+3}\phi_1\|_2
    +\sum_{j=0}^K\|\Lg x\Rg^{j+1}\na^j(\na n_0,n_1)\|_2<\varepsilon,
    \end{split}
\end{equation}}
{where $k\le K-4$,}
then we can conclude the following:

\texttt{(i).} The Cauchy problem for the \eqref{eq:kg} system admits a couple of global solutions $(\phi,n)$ in time with the following uniform energy estimates:  
\begin{align}\label{ue}
\| \partial \Gamma^{\le K}\phi\|_{2}+\| \Gamma^{\le K}\phi\|_{2}+\|\Gamma^{\le K}n\|_{2}\le C\varepsilon
\end{align}
for some constant $C>0$ and $\Ga$ are the usual vector fields in the area of nonlinear waves (see Section 2 for definition). 

\texttt{(ii).} Such global solutions fulfil the following optimal pointwise decay estimate: 
\begin{align}\label{op}
    |\phi(t,x)|\le C_1 \varepsilon \Lg t\Rg^{-\frac32},\quad |n(t,x)|\le C_2 \varepsilon \Lg t\Rg^{-1} \Lg t-|x|\Rg^{-\frac12},
\end{align}
for some positive constants $C_1,C_2>0$ and here $\Lg t\Rg$ denotes the usual Japanese bracket that we refer to Section 2 for the details.

\texttt{(iii).} The solution $(\phi,n)$ scatters to a free solution in $\mathcal{X}_K$ as $t\to+\infty$: there exists $(\phi_{l_0},\phi_{l_1},n_{l_0},n_{l_1})\in \mathcal{X}_K$ such that 
   \begin{align}
    \lim_{t\to \infty}\|(\phi,\partial_t\phi, n,\partial_t n)-(\phi_l,\partial_t\phi_l, n_l,\partial_t n_l)\|_{\mathcal{X}_K}= 0,
\end{align}
where $(\phi_l,n_l)$ is the linear homogeneous solution to the Klein-Gordon-Zakharov system with initial data $(\phi_{l_0},\phi_{l_1},n_{l_0},n_{l_1})$.
\end{thm}

\begin{rem}
There are a few global well-posedeness results for the Klein-Gordon-Zakharov system with small initial data  (cf. \cite{DW20,Kata12,OTT95,D}). In fact our main contribution to this problem is to obtain uniform boundedness of the energy under very mild regularity assumptions on non-compactly supported initial data, namely $K\ge 8$. More specifically, the original work \cite{OTT95} required $K\ge52$ while the paper \cite{D} previously improved the initial condition to $K\ge 15$. The novelty of ours is based on a newly developed normal-type estimate where we manage to handle the lack of the space-time  scaling operator $L_0$. We also provide a proof of the linear scattering as $t\to+\infty$. Another new point of our work lies in giving a clear description of the smallness conditions of non-compactly supported initial data (as mentioned in Remark~\ref{rem1.2}).

\end{rem}

\begin{rem}\label{rem1.2}
 We emphasize that the initial assumption \eqref{ID2} is different from those in nonlinear waves (cf. \cite{CLLX}). More specifically speaking, one can see from \eqref{ID2} that $\|\Lg x\Rg^K \Lg\na\Rg^{K+1}\phi_0\|_2<\varepsilon$ is quite strong since even $\phi_0$ has to obey certain spatial decay properties: $\|\Lg x\Rg^K \phi_0\|_2<\varepsilon$. However such assumption is necessary. To see this, we first observe that when performing energy estimates with vector field $\Omega_{i0}$, one may run into the quantity $\|\Omega_{i0}^4 \phi|_{t=0}\|_2$. Therefore it suffices to consider $\|\Lg x\Rg^4 \partial_t^4\phi|_{t=0}\|_2$. Note that by applying \eqref{eq:kg} twice, one can obtain
 \begin{align}
     \|\Lg x\Rg^4 \partial_t^4\phi|_{t=0}\|_2\lesssim&  \|\Lg x\Rg^4 \partial_t^2\Delta\phi|_{t=0}\|_2+ \|\Lg x\Rg^4 \partial_t^2\phi|_{t=0}\|_2+ \|\Lg x\Rg^4 \partial_t^2(\phi n)|_{t=0}\|_2\\
     \lesssim& \|\Lg x\Rg^4 \Delta^2\phi|_{t=0}\|_2+ \|\Lg x\Rg^4 \Delta \phi|_{t=0}\|_2+ \|\Lg x\Rg^4 \phi|_{t=0}\|_2+\text{nonlinear terms}.
 \end{align}
As shown above, one already needs to control terms such as $\|\Lg x\Rg^4 \phi|_{t=0}\|_2=\|\Lg x\Rg^4 \phi_0\|_2$. We refer the readers to Proposition~\ref{prop} for more details.
 
\end{rem}

\begin{rem}\label{rem1.1}
 To see the decay estimate \eqref{op} is optimal, one can consider the homogeneous linear wave and Klein-Gordon equation in $\R^+\times\R^3$:
\begin{align}
     \Box u=0,\quad \Box v+v=0,
 \end{align}
with sufficiently nice initial data. Then it follows from the fundamental solution formula that the following estimates hold:
\begin{align}
    |u|\lesssim \Lg t\Rg^{-1}\Lg t-r\Rg^{-\frac12},\quad |v|\lesssim\Lg t\Rg^{-\frac32}.
\end{align}
We refer the readers to \cite{H97} for the wave equation and \cite{G1992} for the Klein-Gordon equation.
\end{rem}
\begin{rem}
    It is known from the physics that for strong Langmuir turbulence, the Langmuir phase velocity in the
Klein-Gordon component is of great difference (usually about one thousand times as large ) from the
ion acoustic phase velocity in the wave equation (we refer to \cite{OTT99} and \cite{Plasma}). In fact there is no scaling transformation can make them
equal (while this is possible for the original Zakharov equations), however, the model \eqref{eq:kg} has aroused great attention despite all this in the mathematical context due to its complex nonlinear structure. We refer the readers to the historical review below. Models with more physical background will be discussed in the future.
\end{rem}

We recall some previous results in the literature that are closely related to the presenting paper here.

\subsection*{Historical review}

Nonlinear wave equations have been widely studied: consider the following second order quasilinear wave equation in $[0,\infty)\times\R^d$ $(d\ge2)$:
\begin{equation}\label{eq:wend}
    \Box u=g^{kij}\pa_{k} u\pa_{ij}u,
\end{equation}
where $g^{kij}$ are constants and assume the initial data are sufficiently good. For the $d\ge 4$ case, GWP of \eqref{eq:wend} with small initial data was obtained (cf. \cite{H97}). When $d=3$, GWP of \eqref{eq:wend} was obtained in the pioneering work by Klainerman \cite{K86} and by Christodoulou \cite{C} under the null condition ($g^{kij} \omega_k \omega_i \omega_j=0,$ for $\omega_0=-1$ and $(\omega_1,\omega_2,\omega_3)\in \mathbb{S}^2$). However, on the other hand, the solutions will admit finite time blow-up behaviors if the null conditions do not hold (cf. \cite{J81}). When it comes to the case $d=2$, Alinhac in the seminal work \cite{Alinh01_1,Alinh01_2} shows that \eqref{eq:wend} admits a small data global smooth solution if the null conditions hold and on the contrary the solution blows up in finite time without the null conditions. Moreover, in \cite{Alinh01_1} Alinhac showed that the highest norm of the solution grows at most polynomially in time by introducing the ``ghost weight'' method in order to tackle the slow decay. Many recent work have successfully improved the previous results and we list several of relation here. In \cite{CLX,Li21} the authors prove uniform boundedness of the highest norm of the solution by developing a new normal-form type strategy in \cite{CLXZ}; in \cite{D21} similar results were obtained by applying vector field method on hyperboloids dating back to Klainerman and H\"ormander. Beyond these, there are studies focusing on dealing with models where the Lorentz invariance is not available. Such systems include non-relativistic wave systems with multiple wave speeds (cf. \cite{ ST01}) and 
exterior domains (cf. \cite{M06}). 
In the subsequent work \cite{CLX,CLXZ,Li21}, the authors developed a new systematic normal-form strategy and obtained GWP results of 2D quasilinear wave equations without using the Lorentz boost operators $\Omega_{i0}$; in fact this strategy provides an idea to handle models where space-time scaling operator $L_0$ is not applicable. It is also worth mentioning \cite{CLLX} where the authors obtained GWP of the 2D quasilinear wave system with non-compact small initial data by developing a novel $(L^2, L^\infty)$ estimates relying on the fundamental solution of the wave equation.

The nonlinear Klein-Gordon equations have been roundly studied as well. In the breakthrough work by Klainerman \cite{K85} and Shatah \cite{Shata85}, the Klein-Gordon equations with quadratic nonlinearities were shown to possess small global solutions. Motivated by the large amount of inventive work on nonlinear wave equation and Klein-Gordon equation mentioned above, the coupled wave and Klein-Gordon systems have arouse a great deal of interest for decades. Among which, to our best knowledge, one of the very first result on this area was obtained by Bachelot \cite{B88} on
the Dirac-Klein-Gordon equations. Later in \cite{G90} Georgiev proved GWP with strong null nonlinearities. Much more physical models described by the coupled wave and Klein-Gordon systems
have been studied since then. We list a few models of interest here: the Klein-Gordon-Zakharov equations\cite{D,DW20,Kata12,OTT95}, the
Maxwell-Klein-Gordon equations \cite{KWY20,P99,P99b}, the Einstein-Klein-Gordon equations \cite{LM16,Wang20}, and Dirac-Klein-Gordon model \cite{DLY22}. 

From now on we review some results concerning the Klein-Gordon-Zakharov system arising in plasma physics. We refer the readers to its physical background in \cite{Plasma,Zak72} and briefly review some related models.
Firstly we recall that the Euler-Maxwell equations are of fundamental importance in the area of plasma physics. In particular the Zakharov equations can be derived via the Euler equation for the electrons and ions coupled with the Maxwell equation for the electric field (cf. \cite{Te07}). The Euler-Maxwell equations were shown to admit small data global solutions in $\R^+\times \R^3$ by Guo et.al \cite{GIP16} and the Euler-Poisson system in $\R^+\times \R^2$ was proved to admit
global solutions by Li-Wu \cite{LW14}. 
Secondly, on top of the small data GWP, there are many other interesting results concerning different aspects of the Klein-Gordon-Zakharov system or other closely related models. We hereby list a few of them: GWP of the Klein-Gordon-Zakharov equations with multiple propagation speeds \cite{OTT99}, the convergence of the Klein-Gordon-Zakharov equations to the Schr\"odinger
equation (also the Zakharov system) as certain parameters go to $+\infty$ (cf. \cite{MN08}), the finite time blow-up for Klein-Gordon-Zakharov with rough data \cite{SW19}, the long-term pointwise decay of the small global solution (cf. \cite{D_KG21,DM10,GY95}), scattering (cf. \cite{GNW14, GNW14b, HPS13}). Other related developments with different strategies can be found in the papers  \cite{DM21,DW20, M22,M222,WY14,Y22,Z22}.

The presenting paper was mainly motivated by \cite{D}, where the author proved that the energy for the solution to the Klein-Gordon-Zakharov equations is uniformly bounded. Meanwhile the initial data therein are required to have rather high regularity. To be more precise, the initial condition $(\phi_0,\phi_1,n_0,n_1)$ are supposed to have the following smallness assumption with $K\ge 15$ and $\varepsilon$ being sufficiently small:
\begin{align}\label{1.9}
\sum_{I \leq K+2}
\| \langle x\rangle^{ I +2}\nabla^I \phi_0\| + \sum_{  I \leq K+1}
\| \langle x\rangle^{ I +3}\nabla^I \phi_1\| 
+ \sum_{  I \leq K+1}\| \langle x \rangle^{ I} \nabla^I n_0\| + \sum_{ I\leq K}
\| \langle x \rangle^{ I +1}\nabla^I n_1\| \leq \varepsilon.
\end{align}
Their proof relies on certain $L^\infty-L^\infty$ estimates of both wave and Klein-Gordon components. One of our modest goals is to replace the $L^\infty-L^\infty$ estimates on the wave component together with the contraction mapping by the usual energy bootstrap method (armed with the newly-developed normal form strategy in \cite{CLX, CLXZ}) and eventually reduce the regularity assumption of the initial data to around $K=8$. We also expect to give a clear limn of the assumptions on the initial data. 

Before we sketch the strategy of our proof, we briefly demonstrate the main difficulties of this model problem here. To begin with, as we mentioned earlier in Remark~\ref{rem1.1}, the optimal pointwise decay of $(\phi,n)$ one can expect (ignoring the constant) is \eqref{op}:
\begin{align}\label{1.10}
    |\phi(t,x)|\lesssim \Lg t\Rg^{-\frac32},\quad |n(t,x)|\lesssim \Lg t\Rg^{-1} \Lg t-|x|\Rg^{-\frac12}.
\end{align}
Recall that the Klein-Gordon component
\begin{align}\label{1.11}
    \Box \phi+\phi =-\phi n,
\end{align}
then the usual energy estimate (by taking $L^2$ inner product  with $\partial_t \phi$ on \eqref{1.11}) yields 
\begin{align}
    \frac{d}{dt}(\| \partial \phi\|_2^2+\|\phi\|_2^2)=-2\int_{\R^3}\phi n \cdot\partial_t\phi \ dx,
\end{align}
where $\partial=(\partial_t,\nabla)$ (see Section 2). It is then very natural to integrate in time and estimate the following quantity:
\begin{align}
    \int_0^t \|\phi n\|_2(s) \ ds\lesssim\int_0^t \|\phi\|_{\infty} \|n\|_2 \ ds +\int_0^t \|n\|_{\infty} \|\phi\|_2 \ ds.
\end{align}
Taking \eqref{1.10} into consideration and naively assuming the $L^2$ norms are uniformly bounded in time we arrive at:
\begin{align}
    \int_0^t \|\phi n\|_2 \ ds\lesssim \int_0^t (\Lg s\Rg^{-1}+\Lg s\Rg^{-\frac32})\ ds \lesssim \log(1+t), 
\end{align}
which implies a blow-up of the energy as $t$ goes to $+\infty$. We observe that the decay $\Lg t-r\Rg^{-\frac12}$ has not been made fully use of while taking supremum in space; in fact one of the main advantages of the aforementioned hyperbolic change of variable (vector field method on hyperboloids) is that one can gain better control of the so-called conformal energy thanks to the extra integrability in the hyperbolic time $s=\sqrt{t^2-r^2}$. To deal with such issue, we follow the idea in \cite{D}, namely, apply the well-adapted Alinhac's ghost weight method to absorb extra $\Lg t-r\Rg$ weight and hence generate better time decay. Secondly, the space-time scaling is not invariant in the KGZ system therefore the $L_0$ vector field is not applicable. In the presenting paper, we develop a new normal-form strategy that is motivated by the sequential work \cite{CLX,CLXZ} where the Lorentz invariance is not available. Lastly, another question is whether one can relax the rather strict regularity assumption of the initial data. The assumption \eqref{1.9} in \cite{D} is due to certain $L^\infty-L^\infty$ estimates on the wave equation (with some types of Sobolev embedding). In this paper we employ a new energy bootstrap based on the normal-form strategy so that no $L^\infty-L^\infty$ estimates are needed for the wave component.

\subsection*{Outline of the proof.} We hereby outline the key steps of our proof. To illustrate the idea, we fix any multi-index $\alpha$ and denote $\Phi = \Ga^\al \phi$ and $V=\Ga^\al n$ with $|\al|=K$ (we suppress the dependence on $\alpha$ to ease the notation). We point out here $\Ga$ cannot be $L_0$ due to the lack of space-time scaling invariance; for the exact definition we refer to \eqref{def_Gamma0}--\eqref{def_Gamma}. 

\texttt{Step 1.} {Uniform boundedness of the energy.} To start with, we make the following $a\ priori$ hypothesis:
\begin{align}
|\Gamma^{\le k} n|\le C\varepsilon\langle  t+r\rangle^{-\frac34}\langle r-t\rangle^{-\frac12},\quad |\Gamma^{\le k} \phi|\le C\varepsilon \langle t+r\rangle^{-\frac32},\label{Ab}
\end{align}
where positive constant $C$ is chosen later and $k_1$ is around half size of $K$. Note that here the assumption \eqref{Ab} is not optimal but is still sufficient to derive the uniform boundedness of the energy. In fact, as will be addressed later, we will arrive at the optimal decay estimates in time \eqref{op} as a by-product in the procedure of closing the bootstrap.  We first apply Alinhac's ghost weight method by choosing $p(r,t)=q(r-t)$ with $q'(s)$ scaling almost as $\Lg s\Rg^{-1}$. Let us focus on the Klein-Gordon equation after applying $\Ga^\alpha$:
\begin{align}\label{1.16}
    \Box \Phi+\Phi=-\Ga^\al(\phi n).
\end{align}
Then by taking the $L^2$ inner product on \eqref{1.16} with $\partial_t \Phi\cdot p(r,t)$ we can obtain that
\begin{align}
\frac12\frac{d}{dt}(\| e^\frac{p}{2}\partial \Phi\|_{2}^2+\| e^\frac{p}{2} \Phi\|_{2}^2) +\frac 12
   \int e^p q^{\prime} (|T \Phi|^2+|\Phi|^2)\sim \int e^p (\Phi n +V\phi)\partial_t \Phi\ dx+\cdots,
\end{align}
where ``$\cdots$'' denotes harmless terms which do not contribute to the main term. At first glance, clearly we can control
\begin{align}\label{1.18}
    \left|\int V\phi e^p\partial_t\Phi\ dx\right|\lesssim\|V\|_2\|\phi\|_\infty\|\partial_t\Phi\|_2\lesssim\Lg t\Rg^{-\frac32}\| V\|_2\|\partial_t\Phi\|_2.
\end{align}
The other term is handled by applying the usual Cauchy-Schwartz inequality as below:
\begin{equation}
\left|\int \Phi n e^p\partial_t\Phi\ dx\right|\le \frac14  \int e^p q'\cdot |\Phi|^2\ dx+ \mbox{constant} \cdot \int |\pa \Phi|^2 \cdot\frac{|n|^2 }{q'} e^p\ dx,
\end{equation}
which leads to uniform boundedness of the energy in time {(here  $\frac{|n|^2 }{q'}\sim t^{-\frac32}$ by \eqref{Ab} )}. 

It remains to perform energy estimates on the wave equation as well. As already occurred in \eqref{1.18}, in the
original formulation of the KGZ equations, there appear wave unknowns such as $V=\Ga^\al n$ without derivatives. This leads to difficulties because the undifferentiated quantities cannot be controlled by the natural energy ($\|\partial \Ga^\al n\|_2$). Our approach is to decompose $n$ into two parts: $n = n^0 + \Delta n^\Delta$, where 
\begin{align}
\begin{cases}
\Box n^0=0,\\
(n^0,\partial_t n^0)|_{t=0} =(n_0,n_1),
\end{cases}
\qquad \text{and}\qquad\
\begin{cases}
\Box n^\Delta=|\phi|^2,\\
(n^\Delta,\partial_t n^\Delta) |_{t=0}=0.
\end{cases}
\end{align}
This kind of reformulation can be dated back to \cite{Kata12}, where Katayama considered nonlinear wave equations with
nonlinearities of divergence form. One way to understand this decomposition is to treat $n^\Delta$ as a small perturbation of order $O(\varepsilon^2)$ (the ``leading'' term $n^0$ is of order $O(\varepsilon)$). We refer the readers to Remark~\ref{rem3.2}.

\texttt{Step 2.} {Optimal pointwise time decay and scattering.} We then close the energy bootstrap by showing the optimal pointwise decay \eqref{op} (better decay than \eqref{Ab}). The decay of the Klein-Gordon component $\phi$ follows from a result by Georgiev in \cite{G1992} (cf. Lemma~\ref{kgt32}). In short words, it suffices for us to bound
\begin{align}
    \|\Lg t+|\cdot|\Rg^{1+\delta}\Ga^{\le k+4}(\phi n)\|_2,
\end{align}
which is not hard due to the $a\ priori$ assumption. We also note that Lemma~\ref{lem2.7} plays an important role in dealing with terms such as
\begin{align}
    \|\Lg t+|\cdot|\Rg^{1+\delta}(\Ga^{\le k+4}\phi) n \|_2.
\end{align}
In fact one can view Lemma~\ref{lem2.7} as an iteration process: the Klein-Gordon component can gain better $\Lg t+r\Rg$ decay as long as it can absorb more $\Lg t-r\Rg$ weight.

We then need to obtain the optimal pointwise decay for the wave component $n=n^0+\Delta n^\Delta$. $n^0$ is easier to handle since it is a free wave solution. On the other hand, to estimate the $n^\Delta$ part and remedy the lack of $L_0$ vector fields, we employ $L^{\infty}$ and $L^2$ estimates involving the weight-factor $\langle
r-t\rangle$. The key observation (see Lemma~\ref{lem2.6}) is that one can control $\Lg t-r\Rg\partial^2u$ as follows:
\begin{align}
&|\langle t-r\rangle\partial_{tt} u|+|\langle t-r\rangle\nabla \partial_t u|+|\langle t-r\rangle\nabla^2 u|
 \lesssim  |\partial \Gamma^{\le 1} u|+\langle t+r\rangle|\Box u|.
\end{align}
At the expense of the smallness of the energy, we obtain the optimal decay. These in turn lead to the closure of the energy bootstrap with a careful choice of the constants. In the end, we show the solution $(\phi,n)$ scatters as $t\to+\infty$. In fact this results from a semi-group method and the uniform boundedness of the energy in time.

The rest of this work is organized as follows. In Section 2 we collect the notation and some
preliminaries together with useful lemmas. In Section 3 we show the uniform boundedness of the energy. The optimal pointwise decay in time and scattering are included in Section 4, therefore the proof of Theorem~\ref{thm} is complete. We leave the analysis on the assumption of the initial data in the Appendix.

\section{Preliminaries and notation}

\subsection*{Notation}
We shall use the Japanese bracket notation: $ \langle x \rangle = \sqrt{1+|x|^2}$, for $x \in \mathbb R^3$.  We denote  $\partial_0 = \partial_t$,
$\partial_i = \partial_{x_i}$, $i=1,2,3$ and  
\begin{align}
& \partial = (\partial_i)_{i=0}^3, \; \Omega_{ij}=x_i\partial_j-x_j\partial_i, 1\le i<j\le 3;\; \Omega_{i0}=t\pa_{i}+x_{i}\pa_{t}, 1\le i\le 3;
\; \Omega = r \partial_t + t \partial_r;\notag\\
& \Gamma= (\Gamma_i )_{i=1}^{10}, \quad\text{where } \Gamma_1 =\partial_t, \Gamma_2=\partial_1,
\Gamma_3= \partial_2, \Gamma_4= \partial_3, \Gamma_5=\Omega_{10},\Gamma_6=\Omega_{20},\Gamma_7=\Omega_{30},
\label{def_Gamma0} \\& \Gamma_8=\Omega_{12},\Gamma_9=\Omega_{13}, \Gamma_{10}=\Omega_{23};
\label{def_Gamma}\\
& \Gamma^{\alpha} =\Gamma_1^{\alpha_1} \Gamma_2^{\alpha_2}
\cdots\Gamma_{10}^{\alpha_{10}}, 
\qquad \text{$\alpha=(\alpha_1,\cdots, \alpha_{10})$ is a multi-index};  \notag\\
&\widetilde\Gamma=(\widetilde\Gamma_i )_{i=1}^{11}=(L_0,\Gamma),\qquad L_0=t\partial_t+r\partial_r,\qquad r=|x|; \label{def-tdGamma}\\
& \widetilde\Gamma^{\beta} =\Gamma_1^{\beta_1} \Gamma_2^{\beta_2}
\cdots\Gamma_{11}^{\beta_{11}}, 
\qquad \text{$\beta=(\beta_1,\cdots, \beta_{11})$ is a multi-index};  \notag\\
& T_i = \omega_i \partial_t + \partial_i, \; \omega_0=-1, \; \omega_i=x_i/r, \, i=1,2,3. \label{DefT}
\end{align}
Note that $T_0=0$.
For simplicity of notation,
we define for any integer $k\ge 1$,  $\Gamma^k = (\Gamma^{\alpha})_{|\alpha|=k}$,
$\Gamma^{\le k} =(\Gamma^{\alpha})_{|\alpha|\le k}$.
In particular
\begin{align}
|\Gamma^{\le k} u | = \left(\sum_{|\alpha|\le k} |\Gamma^{\alpha} u |^2\right)^{\frac 12}.
\end{align}
Informally speaking, it is useful to think of  $\Gamma^{\le k} $ as any one of the vector
fields $ \Gamma^{\alpha}$ with $|\alpha| \le k$.

For integer $J\ge 3$, we shall denote 
\begin{align}
E_J = E_J(u(t,\cdot)) = \| (\partial \Gamma^{\le J} u)(t,\cdot) \|_{L_x^2(\mathbb R^3)}^2.
\end{align}

We shall need the following convention for multi-indices: for $\beta=(\beta_1,\cdots,\beta_{10})$ and $\alpha=(\alpha_1,\cdots,\alpha_{10})$, we denote $\beta<\alpha$ if $\beta_i\le\alpha_i $ for $i=1,\cdots,10$ and $|\beta|<|\alpha|$ (Here $|\alpha|=\sum_{i=1}^{10}\alpha_i$). Similarly we denote $\beta\le \alpha$ if $\beta_i\le\alpha_i $ for $i=1,\cdots,10$.

For any two quantities $A$, $B\ge 0$, we write  $A\lesssim B$ if $A\le CB$ for some unimportant constant $C>0$ and such $C$ may vary from line to line if not specified.
We write $A\sim B$ if $A\lesssim B$ and $B\lesssim A$. We write $A\ll B$ if
$A\le c B$ and $c>0$ is a sufficiently small constant. The needed smallness is clear from the context.

{Throughout this work we assume $t\ge 2$ if it is not specified.}

\subsection*{Decay estimates}
We collect some decay estimates in the following lemmas.

\begin{lem}[Klainerman--Sobolev]\label{KSineq}
Let $h(t,x)\in C^\infty([0,\infty)\times \mathbb{R}^3)$ and $ h(t,x)\in \mathcal{S}(\R^3)$  for every $t>0$. Then
\begin{align}\label{ksineqa}
\Lg t+|x|\Rg\Lg t-|x|\Rg^{\frac12}|h(t,x)|\lesssim  \Vert \widetilde\Gamma^{\leq 3} h(t,\cdot) \Vert_{2}\qquad \forall t>0, \ x\in\mathbb{R}^3.
\end{align}

\end{lem}
\begin{proof}
    The proof is standard. One can find it in Lemma 2.3 of \cite{CLLX} for example. We also emphasize that \eqref{ksineqa} involves $L_0$ that may lead to trouble when applied to the KGZ system, so that one must be carefully examining $L_0$ terms. 
\end{proof}

\begin{lem}
For $x\in \R^3$ with $|x|\le\frac t2$, $t\ge 1$, we have for $u\in \mathcal{S}(\R^3)$ 
\begin{align}\label{t34}
t^{\frac 34}|u(t,x)|\lesssim \|\partial^{\le 1}\Gamma^{\le 2} u\|_{L_x^2(\R^3)}
\end{align}
where $\pa=\pa_t,\pa_1,\cdots,\pa_3$.
\end{lem}
\begin{proof}
We refer the readers to Lemma 2.4 in \cite{G1992}.
\end{proof}

\begin{lem}

For $|x|\ge \frac t2$ and $t\ge 1$, we have the following estimate for $u\in \mathcal{S}(\R^3)$:
\begin{align}\label{t1}
\langle t-|x|\rangle^{\frac12}\langle t+|x|\rangle|u(t,x)|\lesssim \|\Gamma^{\le 2} u\|_{2}+\|\langle t-|\cdot|\rangle\nabla \Gamma^{\le 2}u\|_2.
\end{align}
\end{lem}

\begin{proof}
We write $u(t,x)$ in the spherical coordinates as $v(t,r,\theta,\varphi)$ with $x=(r\cos\theta\cos\varphi,r\cos\theta\sin\varphi,r\sin\theta)$. Then we have 
$|u|\lesssim \|\partial_{\theta}^{\le 1}\partial_{\varphi}^{\le 1} v\|_{L_{\theta}^2L_{\varphi}^2}$.
If $|t-r|\le 1$, we have 
\begin{align}
\langle t-r\rangle t^2|u|^2\lesssim t^2|u(t,x)|^2
\lesssim&\ t^2\int_{r}^{\infty}|\partial_\rho(v(t,\rho,\theta,\varphi)^2)|\ d\rho
\\\lesssim& \int_{r}^{\infty}|v(t,\rho,\theta,\varphi)||\partial_\rho v(t,\rho,\theta,\varphi)|\ \rho^2d\rho\qquad (\text{by } r\ge\frac t2)
\\\lesssim& \|\partial_{\theta}^{\le 1}\partial_{\varphi}^{\le 1} v\|_{2}^2+\|\partial_{\theta}^{\le 1}\partial_{\varphi}^{\le 1}\partial_\rho  v\|_{2}^2
\\\lesssim& \|\Gamma^{\le 2}u\|_{2}^2+\|\nabla \Gamma^{\le 2}u\|_{2}^2.
\end{align}
Now we assume $|t-r|\ge 1$. We use the Sobolev inequality: for any $h\in \mathcal S(\R)$, we have for $\frac t2\le r< t$
\begin{align}
(t-r)r^2|h(r)|^2\lesssim &\int_{r}^t|\partial_\rho \left((t-\rho)\rho^2|h(\rho)|^2\right)|d\rho
\\ \lesssim&\int_{r}^t|h(\rho)|^2 \rho^2d\rho+\int_{r}^t|t-\rho||h(\rho)|^2\rho d\rho+\int_{r}^t |t-\rho|\rho^2|\partial_\rho h(\rho)||h(\rho)|d\rho
\\ \lesssim&\int_{r}^t|h(\rho)|^2 \rho^2d\rho+\int_{r}^t|t-\rho|^2|\partial_\rho h(\rho)|^2 \rho^2d\rho.
\end{align}
For $r>t$, one gets
\begin{align}
(r-t)r^2|h(r)|^2\lesssim &\int_{t}^r|\partial_\rho \left((\rho-t)\rho^2|h(\rho)|^2\right)|d\rho
\\ \lesssim&\int_{t}^r|h(\rho)|^2 \rho^2d\rho+\int_{t}^r|t-\rho||h(\rho)|^2\rho d\rho+\int_{t}^r |t-\rho|\rho^2|\partial_\rho h(\rho)||h(\rho)|d\rho
\\ \lesssim&\int_{t}^r|h(\rho)|^2 \rho^2d\rho+\int_{t}^r|t-\rho|^2|\partial_\rho h(\rho)|^2 \rho^2d\rho.
\end{align}
These imply that (after applying $|u|\lesssim \|u\|_{H^2(\mathbb{S}^2)}$)
\begin{align}
\langle t-r\rangle^{\frac12} t|u(t,x)|\lesssim \|\Gamma^{\le 2}u\|_{2}+\|\langle t-|\cdot|\rangle\nabla \Gamma^{\le 2}u\|_{2}.
\end{align}
\end{proof}

\begin{lem}\label{lem2.6}
Suppose $u=u(t,x)$ has continuous second order derivatives. Then for  $t>0$ and $r=|x|$ we have
\begin{align}\label{2.260}
&|\langle t-r\rangle\partial_{tt} u|+|\langle t-r\rangle\nabla \partial_t u|+|\langle t-r\rangle\nabla^2 u|
 \lesssim  |\partial \Gamma^{\le 1} u|+\langle t+r\rangle|\Box u|.
\end{align}
\end{lem}

\begin{proof}
It is clear that \eqref{2.260} holds for $|t-r|\le 1$. We then assume $|t-r|\ge 1$.

\texttt{Case 1. $r\le 2t$.} 
Recalling $\Omega_{i0}u=x_i\partial_t u+t\partial_i u$, we have for $x\in\R^3$
\begin{align}\label{2.290}
&\partial_t\Omega_{i0}u= x_i\partial_{tt}u+\partial_i u+t\partial_t\partial_i u,\\
&\partial_i\Omega_{i0}u= 3\partial_{t}u+r\partial_r\partial_t u+t\Delta u.  
\end{align}
(Note we used Einstein summation here).
This implies that 
\begin{align}
&(t^2-r^2)\Delta u =r^2\Box u+r\partial_r u-3t\partial_t u-x_i\partial_t\Omega_{i0}u+t\partial_i\Omega_{i0}u\qquad(\text{by }\Box =\partial_{tt}-\Delta), \label{2.300}\\
\implies & |(t-r) \Delta u| \lesssim r|\Box u|+|\partial\Gamma^{\le 1} u|,\\
\implies &|( t-r)\partial_{tt} u| \le |( t-r)(\Box u+\Delta u)|
\lesssim  (t+r)|\Box u|+|\partial\Gamma^{\le 1} u|.\label{2.330}
\end{align}
Then we focus on $\partial_{t}\nabla u$ and $\nabla^2 u$.  From \eqref{2.290} we know that for $1\le i\le 3$
\begin{align}
&\partial_i \partial_tu=\frac{1}{t}\partial_t \Omega_{i0}u-\frac1t\partial_i u-\frac{x_i}t\partial_{tt} u,\\
\implies &|(t-r)\partial_i \partial_tu|\lesssim|\partial_t \Omega_{i0}u|+|\partial_i u|+|(t-r)\partial_{tt} u|\lesssim |\partial\Gamma^{\le 1}u|+(t+r)|\Box u| \qquad (\text{by \eqref{2.330}}).\label{2.350}
\end{align}
By the definition of $\Omega_{i0}u=x_i\partial_{t}u+t\partial_i u$, we have for $1\le i,j\le 3$
\begin{align}
\partial_{j}\Omega_{i0}u=&\ \delta_{ij}\partial_{t}u+x_{i}\partial_{t}\partial_{j}u+t\partial_{ij} u,
\\ \implies |(t-r)\partial_{ij}u|=&\frac {(t-r)}t( \partial_{j}\Omega_{i0}u-\delta_{ij}\partial_{t}u-x_{i}\partial_{t}\partial_{j}u)\lesssim |\partial \Gamma^{\le 1}u|+|(t-r)\partial_{t}\partial_ju|
\\ \lesssim &|\partial \Gamma^{\le 1}u|+(t+r)|\Box u|\qquad(\text{by }\eqref{2.350}).
\end{align}
\texttt{Case 2. $r\ge \frac t2$}.  Recall that $\Omega=t\partial_r+r\partial_t$. Therefore we have
\begin{align}
\Omega \partial_{t}u=t\partial_r\partial_{t}u+r\partial_{tt}u,\qquad
\Omega \partial_{r}u=t\partial_{rr}u+r\partial_t\partial_{r}u.
\end{align}
Then we have 
\begin{align}
(t^2-r^2)\partial_t\partial_r u=&\ t\Omega\partial_t u-r\Omega\partial_r u-tr(\partial_{tt}u-\partial_{rr}u)
\\=&\ t\Omega\partial_t u-r\Omega\partial_r u-tr(\Box u+\frac 2 r\partial_{r}u+\frac1{r^2}\Delta_{\mathbb{S}^2} u) \quad (\text{ by }\Delta =\partial_{rr}+\frac 2r\partial_r+\frac{\Delta_{\mathbb{S}^2}}{r^2}),
\\ \implies |(t-r)\partial_{t}\partial_r u|\lesssim& |\partial\Gamma^{\le 1} u|+(t+r)|\Box u|.
\end{align}
Here $\Delta_{\mathbb{S}^2}$ is the Laplace-Beltrami operator on the sphere.
{Therefore we get $|( t-r)\partial_{t}\nabla u|\lesssim |\partial\Gamma^{\le 1} u|+(t+r)|\Box u|$.}
On the other hand, we have 
\begin{align}
&t^2\partial_{rr}u-r^2\partial_{tt}u=t\Omega\partial_r u-r\Omega\partial_t u,
\\ \implies& (t^2-r^2)\partial_{rr}u=2r\partial_r u+\Delta_{\mathbb S^2}u+r^2\Box u+t\Omega\partial_r u-r\Omega\partial_t u,
\\ \implies&|( t-r)\partial_{rr}u|\lesssim |\partial\Gamma^{\le 1} u|+(t+r)|\Box u|. \label{2.420}
\end{align}
This implies that $\langle t-r\rangle|\Delta u|$ and $\langle t-r\rangle|\nabla^2 u|$ can be controlled by RHS of \eqref{2.420}. Furthermore we use the identity $\partial_{tt}u =\Delta u+\Box u$ and deduce that $\langle t-r\rangle|\partial_{tt} u|\lesssim |\partial\Gamma^{\le 1} u|+(t+r)|\Box u|$. Thus we complete the proof.

\end{proof}

\begin{lem}
Assume $h\in \mathcal{S}(\R^3)$. Then
\begin{align}\label{S-2}
  \|h\|_{2}\lesssim \|\Lg x\Rg\nabla h\|_{2}.
\end{align}
\end{lem}

\begin{proof}The proof is similar to the 2D version as in \cite{CLLX}.
Let $\varphi=\frac{r}{\Lg r\Rg}$. Then
\begin{align*}
&0\leq \int (\Lg r\Rg \pa_{r}h+\varphi h)^2\ dx=\int \left(\Lg r\Rg^2(\pa_{r}h)^2+\Lg r\Rg\varphi\pa_{r}(h^2)+\varphi^2h^2\right) dx,
\\ \implies&\int \Lg r\Rg^2 |\pa_rh|^2\ dx \geq \int \left(\frac{1}{r^2}\pa_{r}(r^2\Lg r\Rg \varphi)-\varphi^2\right)h^2 \ dx\geq \int (3-(\frac{r}{\Lg r\Rg})^2)h^2\ dx\gtrsim \int |h|^2\ dx.
\end{align*}

\end{proof}

\subsection*{Pointwise estimates for the Klein-Gordon component}
We here collect some previously known pointwise estimates for the Klein-Gordon components, one can refer to Georgiev in \cite{G1992} (also see \cite{D}).
Let $\{p_j\}_{0}^\infty$ be the usual Littlewood-Paley partition of the unity
\[\sum_{j\ge 0} p_j(s)=1,\qquad s\ge 0,\]
which satisfies 
\begin{align}
0\le p_j\le 1,\qquad p_j\in C_0^\infty(\R)\qquad \text{for all }j\ge 0, 
\end{align}
and 
\begin{align}
\supp\ p_0\subset(-\infty, 2],\qquad \supp\ p_j\subset[2^{j-1},2^{j+1}]\qquad \text{for all }j\ge1.
\end{align}
\begin{lem}\label{kgt32}
Let $w$ be the solution of the Klein-Gordon equation
\begin{align}
\begin{cases}
\Box w+w=f,\\
 (w,\partial_t w)|_{t=0} =(w_0,w_1),
\end{cases}
\end{align}
with $f=f(t,x)$ a sufficiently nice function. Then for all $t\ge 0$, it holds
\begin{align}
\langle t+|x|\rangle^{\frac32}|w(t,x)|\lesssim \sum_{j\ge 0}(\sup_{0< s\le t}p_j(s)\|\langle s+|\cdot|\rangle\Gamma^{\le 4}f(s,\cdot)\|_{L^2}+\|\langle |\cdot|\rangle^{\frac 32} p_j(|\cdot|)\Gamma^{\le 4}w(0,\cdot)\|_{L^2}).
\end{align}

\end{lem}
\begin{lem}
With the same settings as in Lemma \ref{kgt32}, then for all $t\ge 0$ 
\begin{align}\label{d-1}
\langle t+|x|\rangle^{\frac32}|w(t,x)|\lesssim \sup_{0< s\le t}\|\langle s+|\cdot|\rangle^{1+\delta}\Gamma^{\le 4}f(s,\cdot)\|_{L^2} +\|{\langle |\cdot|\rangle^{\frac32+\delta }}\Gamma^{\le 4}w(0,\cdot)\|_{L^2}.
\end{align}
Here $0<\delta\ll1$.

\end{lem}

\begin{lem}\label{lem2.7}
Suppose $u=u(t,x)$ is a smooth solution to $\Box u+u=F$. Then for $t>0$ we have 
\begin{align}\label{kgw}
\Big|\frac{\langle t+r\rangle}{\langle t-r\rangle}u\Big|
\lesssim&|\partial\Gamma^{\le 1}u|+\Big|\frac{\langle t+r\rangle}{\langle t-r\rangle}F\Big|.
\end{align}
\end{lem}

\begin{proof}
If $r< \frac t2$ or $r> 2t$, we have $\langle t+r\rangle \sim\langle t-r\rangle$. This implies that 
\[\Big|\frac{\langle t+r\rangle}{\langle t-r\rangle}u\Big|
\lesssim| u|\lesssim |\Box u|+|F|
\lesssim |\partial\Gamma^{\le 1}u| +|F|.\]
If $\frac t2\le r\le 2t$, we recall that \eqref{2.300} gives that
\begin{align}
(t^2-r^2)\Delta u =&r^2\Box u+r\partial_r u-3t\partial_t u-x_i\partial_t\Omega_{i0}u+t\partial_i\Omega_{i0}u
\\ =& -r^2  u+r^2F+r\partial_r u-3t\partial_t u-x_i\partial_t\Omega_{i0}u+t\partial_i\Omega_{i0}u\qquad (\text{by }\Box u+u=F)
\\ \implies u=&-\frac{t^2-r^2}{r^2}\Delta u+F+\frac1r\partial_r u-\frac {3t}{r^2}\partial_t u-\frac1{r^2}(x_i\partial_t\Omega_{i0}u-t\partial_i\Omega_{i0}u).
\end{align}
This implies \eqref{kgw} holds for $\frac t2\le r\le 2t$. 
\end{proof}
\begin{rem}
  Note that one can iterate \eqref{kgw} as long as the nonlinear term $F$ is sufficiently nice. One example is $F=u^2$ then $F$ can easily absorb $\frac{\Lg t+r\Rg}{\Lg t-r\Rg}$ and eventually lead to the decay of $|u|\lesssim \Lg t+r\Rg^{-M}\Lg t-r\Rg^M$ for any large $M$. Such property of $u$ can be considered as a machinery that one can gain any decay in $\Lg t+r\Rg$ with a trade-off of losing $\Lg t-r\Rg$ decay.
\end{rem}

\subsection*{Verification of the smallness conditions of the Initial data}
In this subsection we verify the smallness condition \eqref{ID2} is consistent with the vector fields:
\begin{prop}\label{prop}
Assume $(\phi_0,\phi_1,n_0,n_1)$ satisfy the condition \eqref{ID2}, then for any $J\le K$ we have
    \begin{align}
     &\|\Ga^{\le J}\phi|_{t=0}\|_2\lesssim \big(\|\Lg x\Rg^J \Lg \na\Rg^J  \phi_0\|_2+\|\Lg x\Rg^J  \Lg \na \Rg^{J-1}\phi_1\|_2\big)\big(1+\sum_{j=0}^{J-3}\|\Lg x\Rg^{j+1}\na^j (\na n_0,n_1)\|_2\big),\label{B1}\\
     &\|\Ga^{\le J} n|_{t=0}\|_2\lesssim \big(\|\Lg x\Rg^{J-2} \Lg \na\Rg^{J-2}  \phi_0\|_2+\|\Lg x\Rg^{J-2}  \Lg \na \Rg^{J-3}\phi_1\|_2\big)\big(1+\sum_{j=0}^{J-1}\|\Lg x\Rg^{j+1}\na^j (\na n_0,n_1)\|_2\big).\label{B2}
    \end{align}
Similarly, we also have
\begin{align}\label{B3}
    \|\Lg x\Rg^{2}\Ga^{\le k+4}\phi|_{t=0}\|_2\lesssim\big(\|\Lg x\Rg^{k+6} \Lg \na\Rg^{k+4}  \phi_0\|_2+\|\Lg x\Rg^{k+6}  \Lg \na \Rg^{k+3}\phi_1\|_2\big)\big(1+\sum_{j=0}^{k+2}\|\Lg x\Rg^{j+1}\na^j (\na n_0,n_1)\|_2\big).
\end{align}

\end{prop}
\begin{proof}
    The proof relies on an induction that is very similar to \cite{CLLX}, therefore we only sketch the idea. With no loss, we only consider the ``worst'' scenario as follows:
\begin{align}
    \|\Lg x\Rg^J \partial_t^J\phi|_{t=0}\|_2\lesssim \|\Lg x\Rg^J \partial_t^{J-2}\Delta \phi|_{t=0}\|_2+\|\Lg x\Rg^J \partial_t^{J-2}\phi|_{t=0}\|_2+\|\Lg x\Rg^J\partial_t^{J-2}(\phi n)|_{t=0}\|_2.
\end{align}
Note that by induction assumption we shall have
\begin{align}
    \|\Lg x\Rg^J \partial_t^{J-2}\Delta\phi|_{t=0}\|_2+\|\Lg x\Rg^J \partial_t^{J-2}\phi|_{t=0}\|_2\lesssim
\| \Lg x\Rg^{J}\Lg \na\Rg^J\phi_0\|_2+\| \Lg x\Rg^{J}\Lg \na\Rg^{J-1}\phi_1\|_2.
\end{align}
For the nonlinear part, naively we only present the worst case:
\begin{align}
    \|\Lg x\Rg^J\partial_t^{J-2}(\phi n)|_{t=0}\|_2\lesssim&\|\Lg x\Rg^J(\partial_t^{\le  J-2}\phi)( \partial_t^{\le \frac{J-2}{2}}n)|_{t=0}\|_2+\|\Lg x\Rg^J(\partial_t^{\le J-2}n)( \partial_t^{\le \frac{J-2}{2}}\phi)|_{t=0}\|_2+\cdots\\
    \lesssim&\|\Lg x\Rg^J\partial_t^{\le J-2}\phi |_{t=0}\|_2\|\partial_t^{\le \frac{J-2}{2}} n|_{t=0}\|_\infty+\|\Lg x\Rg^J\partial_t^{\le \frac{J-2}{2}}\phi |_{t=0}\|_\infty\|\partial_t^{\le {J-2}} n|_{t=0}\|_2+\cdots,
\end{align}
where $\cdots$ are terms that can be handled similarly. Then by induction and the $H^2(\R^3)\hookrightarrow L^\infty(\R^3)$ embedding we have
\begin{align}
&\|\Lg x\Rg^J\partial_t^{\le J-2}\phi |_{t=0}\|_2\|\partial_t^{\le \frac{J-2}{2}} n|_{t=0}\|_\infty\\
\lesssim& (\| \Lg x\Rg^{J}\Lg \na\Rg^{J-2}\phi_0\|_2+\| \Lg x\Rg^{J}\Lg \na\Rg^{J-3}\phi_1\|_2)\sum_{j=0}^{\frac{J}{2}}\|\Lg x\Rg^{j+1}\na^j(\na n_0,n_1)\|_2
\end{align}
and 
\begin{align}
&\|\Lg x\Rg^J\partial_t^{\le \frac{J-2}{2}}\phi |_{t=0}\|_\infty\|\partial_t^{\le {J-2}} n|_{t=0}\|_2\\
\lesssim& (\| \Lg x\Rg^{J}\Lg \na\Rg^{\frac{J+2}{2}}\phi_0\|_2+\| \Lg x\Rg^{J}\Lg \na\Rg^{\frac{J}{2}}\phi_1\|_2)\sum_{j=0}^{J-3}\|\Lg x\Rg^{j+1}\na^j(\na n_0,n_1)\|_2.
\end{align}
The general case $\|\Gamma^{J} \phi|_{t=0}\|_2$ follows from similar arguments.

\end{proof}

\section{uniform boundedness of the energy}
To prove Theorem~\ref{thm},
we first make an {\it a priori} hypothesis: assume 
\begin{align}
&|\Gamma^{\le k} n|\le C_1\varepsilon\langle t+r\rangle^{-\frac34}\langle r-t\rangle^{-\frac12},\label{A2}
\\&|\Gamma^{\le k} \phi|\le C_2\varepsilon \langle t+r\rangle^{-\frac32},\label{A3}
\end{align}
where the positive constants $C_1,\ C_2$ will be chosen later. As mentioned earlier, we decompose the wave component into
$n=n^0+\Delta n^\Delta$, where 
\begin{align}\label{3.333}
\begin{cases}
\Box n^0=0,\\
(n^0,\partial_t n^0)|_{t=0} =(n_0,n_1),
\end{cases}
\qquad \text{and}\qquad\
\begin{cases}
\Box n^\Delta=|\phi|^2,\\
(n^\Delta,\partial_t n^\Delta) |_{t=0}=\mathbf 0.
\end{cases}
\end{align}

\begin{lem}[Uniform boundedness of the energy]\label{lem3.1}
Let $(\phi,n)$ be a couple of solutions to \eqref{eq:kg}. Assume \eqref{A2} and \eqref{A3} hold. Let $K\ge k+4$.
Then we have 
\begin{align}\label{3.32}
\| \partial \Gamma^{\le K}\phi\|_{2}+\| \Gamma^{\le K}\phi\|_{2}+\|\partial\nabla\Gamma^{\le K} n^\Delta\|_{2}+\|\Gamma^{\le K}n\|_{L^2}\le C\varepsilon
\end{align}
and 
\begin{align}\label{3.34}
\|\pa\Gamma^{\le k+4} n^{\Delta}\|_{L^2} \le (C\varepsilon)^2.
\end{align}

\end{lem}

\begin{proof} To begin with we first consider the wave component $\|\Ga^{\le K}n\|_2$.

\texttt{Step 1, the wave component $n$.}

Note that since $n=n^0+\Delta n^\Delta$, we have for any multi-index $\alpha$,
\begin{align}\label{3.31}
|\Gamma^{\alpha}\Delta n^\Delta|\lesssim |\Delta \Gamma^{\le |\alpha|}n^\Delta|+ |\nabla \partial_t \Gamma^{\le |\alpha|}n^\Delta|\lesssim |\partial\nabla \Gamma^{\le |\alpha|} n^{\Delta}|.
\end{align}
As a result for $|\beta|\le K$, it is clear that
\begin{align}\label{3.21}
\|\Gamma^{\beta}n\|_{2}^2\lesssim \|\Gamma^{\beta}n^0\|_{2}^2+\|\Gamma^{\beta}\Delta n^{\Delta}\|_{2}^2
\lesssim \|\Gamma^{\beta}n^0\|_{2}^2+\|\partial\nabla\Gamma^{\le |\beta|}n^{\Delta}\|_{2}^2.
\end{align}
Then we estimate the two terms in RHS of \eqref{3.21}.

\texttt{ Step 1.1, we estimate $\|\Gamma^{\beta}n^0\|_{L^2}^2$.}
Denote $X(\partial)=(1+t^2+r^2)\partial_t+2tr\partial_r$. Obviously,
\begin{align}
X(\partial)u=&(1+t^2+r^2)\partial_tu+2tr\partial_ru  
\\ =&\ \partial_tu+t^2\partial_tu+tr\partial_ru+r^2\partial_tu+tr\partial_ru  
\\=&\ \partial_tu +tL_0u+r\Omega u  
\\ \lesssim &\langle r+t\rangle(|L_0 u|+|\Gamma u|). 
\end{align}
Let $|\tilde\beta|\le K-1$, we have $\Box \Gamma^{\tilde\beta} n^0=0$. Then we have 
\begin{align}
0=\int \Box \Gamma^{\tilde\beta} n^0 \cdot X(\partial)\Gamma^{\tilde\beta} n^0\sim\frac d{dt}(\|\Gamma^{\le 1} \Gamma^{\tilde\beta} n^0\|_2^2+\| L_0 \Gamma^{\tilde\beta} n^0\|_2^2).
\end{align}
This implies that 
\begin{align}\label{3.13}
\| \Gamma^{\le K} n^0\|_2^2\lesssim \sum_{|\tilde \beta|\le K-1}\|\Gamma \Gamma^{\tilde\beta} n^0\|_2^2(0)+{\|L_0\Gamma^{\tilde\beta} n^0\|_2^2(0)+\|\Ga^{\tilde\beta}n^0\|_2^2(0)}{\le (C \varepsilon)^2.}
\end{align}
{Note that here the smallness condition on the initial data \eqref{ID2} is required:
\begin{align}\label{3.12a}
   \|(\Gamma^{\le K} n^0)_{t=0}\|_{2}+\| (L_0 \Gamma^{\le K-1} n^0)_{t=0}\|_2^2\le \sum_{j=0}^{K}\|\langle x\rangle^{j}\nabla^j n_0\|_{2}+\sum_{j=0}^{K-1}\|\langle x\rangle^{j+1}\nabla^j n_1\|_{2}\le C_i\varepsilon, 
\end{align}
where the constant $C_i>0$ is known from \eqref{B1}-\eqref{B3}. We also point out here that the requirement of $\|n_0\|_2\le \varepsilon$ can be replaced by $\|\langle x\rangle \nabla n_0\|_2$ as a consequence of the embedding in \eqref{S-2}.}

\begin{rem}
This method is called $X(\partial)$ trick in some literature (cf. \cite{H97}) and known as the conformal energy estimates in other references such as \cite{D,D21}. Since the equation of $n^0$ is not related to $\phi$, then we can replace $\Gamma$ by $\widetilde\Gamma$ in \texttt{step 1.1}. Thus we obtain the following estimates:
\begin{align}\label{3.130}
\| \widetilde\Gamma^{\le K} n^0\|_2\lesssim \|\widetilde \Gamma^{\le K} n^0\|_2(0)
\lesssim \sum_{j=0}^{K-1}\|\langle x\rangle^{j+1}\nabla^j (\nabla n_0,n_1)\|_{2}\le \varepsilon.
\end{align}
\end{rem}

\texttt{Step 1.2, we estimate $\|\partial \nabla\Gamma^{\le |\beta|}n^{\Delta}\|_{2}^2$.}
 Let $|\gamma|=m$ and $m\le |\beta|$ is a running index, we have 
\begin{align}\label{3.17}
\Box \nabla\Gamma^\gamma n^\Delta=\sum_{\gamma_1+\gamma_2= \gamma}C_{\gamma;\gamma_1,\gamma_2}\nabla\Gamma^{\gamma_1}\phi\Gamma^{\gamma_2}\phi.
\end{align}
Multiplying both sides of \eqref{3.17} by $\partial_t\nabla\Gamma^\gamma n^\Delta$, we obtain
\begin{align}
\frac{d}{dt}\|\partial\nabla\Gamma^\gamma n^\Delta\|_{2}^2
\lesssim& \int|\nabla\Gamma^{\gamma_1}\phi\Gamma^{\gamma_2}\phi \partial_t\nabla\Gamma^\gamma n^\Delta|
 \\ \lesssim &(\|\nabla\Gamma^{\le [\frac{K-1}2]}\phi\|_{\infty}\|\Gamma^{\le K}\phi\|_{2}
              +\|\nabla\Gamma^{\le K}\phi\|_{2}\|\Gamma^{\le [\frac{K}2]}\phi\|_{\infty})
              \| \partial_t\nabla\Gamma^\gamma n^\Delta\|_{2}
 \\ \lesssim &  \|\Gamma^{\le [\frac{K-1}2]+1}\phi\|_{\infty} (\|\nabla\Gamma^{\le K}\phi\|_{2}+\|\Gamma^{\le K}\phi\|_{2})  \| \partial_t\nabla\Gamma^\gamma n^\Delta\|_{2}         
 \\ \le & \ C\varepsilon t^{-\frac32}(\|\nabla\Gamma^{\le K}\phi\|_{2}+\|\Gamma^{\le K}\phi\|_{2})\| \partial_t\nabla\Gamma^\gamma n^\Delta\|_{2} \qquad\qquad (\text{by \eqref{A3}}).
\end{align}
Note that here we need $[\frac{K-1}{2}]+1\le k\le K-4 \implies K\ge 8$.
This implies that 
\begin{align}\label{3.27}
\frac{d}{dt}\|\partial\nabla\Gamma^\gamma n^\Delta\|_{2}^2
\le & \ C\varepsilon t^{-\frac32}(\|\nabla\Gamma^{\le K}\phi\|_{2}^2+\|\Gamma^{\le K}\phi\|_{2}^2+\| \partial_t\nabla\Gamma^\gamma n^\Delta\|_{2}^2).
 \end{align}

\texttt{Step 2, we estimate the Klein-Gordon component $\|\Ga^{\le K}\phi\|_2, \|\pa\Ga^{\le K}\phi\|_2$.}

Let $\alpha$ be a multi-index and $|\alpha|\le K$. 
Write $\Phi= \Gamma^{\alpha} \phi$. We have
\begin{align}\label{3.0A}
\square \Phi+\Phi=-\sum_{\alpha_1+\alpha_2=\alpha}C_{\alpha; \alpha_1,\alpha_2} \Gamma^{\alpha_1}n \Gamma^{\alpha_2}\phi.
\end{align}
Let $p(t,r) = q(r-t)$, where
\begin{align}\label{q}
q(s) = \int_{-\infty}^s \langle \tau\rangle^{-1}  \bigl(\log ( 2+\tau^2) \bigr)^{-2} d\tau.
\end{align}
Clearly
\begin{align}
-\partial_t p = \partial_r p = q^{\prime}(r-t)
= \langle r -t \rangle^{-1} \bigl( \log (2+(r-t)^2)  \bigr)^{-2}.
\end{align}
Multiplying both sides of \eqref{3.0A} by $e^p \partial_t \Phi$, we obtain
\begin{align}
   \text{LHS}&=\frac12\frac{d}{dt}(\| e^\frac{p}{2}\partial \Phi\|_{2}^2+\| e^\frac{p}{2} \Phi\|_{2}^2) +\frac 12
   \int e^p q^{\prime} (|T \Phi|^2+|\Phi|^2),
\end{align}
where $|\partial \Phi|^2=\sum_{i=0}^3|\pa_i \Phi|^2$ and $|T\Phi|^2=\sum_{i=1}^3|T_i \Phi|^2$. We split the RHS into two cases:
\begin{align}
\text{RHS} &= -\sum_{\alpha_1+\alpha_2=\alpha}C_{\alpha; \alpha_1,\alpha_2} \int e^p \Gamma^{\alpha_1}n \Gamma^{\alpha_2}\phi \partial_t \Phi.
\end{align}

\texttt{Case 1:} $\alpha_1<\alpha_2\le \alpha$.
\begin{align}
\text{RHS}\lesssim \int e^p q' |\Gamma^{\alpha_2}\phi|^2+\int \frac{e^p}{q'}|\Gamma^{\le [\frac{|\alpha|}2]}n|^2|\partial_t \Phi|^2\lesssim&\OK+\|\langle r-t\rangle^{\frac{1+\delta_1}2} \Gamma^{\le [\frac{|\alpha|}2]}n\|_{\infty}^2\|\partial_t\Phi\|_{2}^2
\\\le& \OK+(C\varepsilon)^2t^{-\frac32+\delta_1}\|\partial_t\Phi\|_{2}^2\qquad (\text{by \eqref{A2}}).
\end{align}
Here $0< \delta_1<\frac12$ and  
OK are the harmless terms that can be absorbed by the LHS.

\texttt{Case 2:} $\alpha\ge\alpha_1\ge \alpha_2$.
\begin{align}
\text{RHS}\lesssim &\| \Gamma^{\le|\alpha|}n\|_{2}\|\Gamma^{\le [\frac{|\alpha|}2]}\Phi\|_{\infty}\|\partial_t\Phi\|_{2}
\\ \le &C\varepsilon t^{-\frac32}(\| \Gamma^{\le|\alpha|}n^0\|_{2}+\| \Gamma^{\le|\alpha|}\Delta n^\Delta\|_{2})\|\partial_t\Phi\|_{2}\qquad (\text{by \eqref{A3} and $n=n^0+\Delta n^\Delta$})
\\ \le &(C\varepsilon)^2 t^{-\frac32}\|\partial_t\Phi\|_{2}+C\varepsilon t^{-\frac32}\| \partial \nabla\Gamma^{\le |\alpha|} n^\Delta\|_{2}\|\partial_t\Phi\|_{2} \qquad (\text{by \eqref{3.13}}).
\end{align}
Then we derive that 
\begin{equation}\label{3.29}
\begin{split}
&\frac{d}{dt}(\| \partial \Phi\|_{2}^2+\| \Phi\|_{2}^2) +\frac 12
   \int e^p q^{\prime} (|T \Phi|^2+|\Phi|^2)\\\le&\  (C\varepsilon)^2t^{-\frac32+\delta_1}\| \partial \Phi\|_{2}^2+(C\varepsilon)^2t^{-\frac32}\| \partial \Phi\|_{2}+C\varepsilon t^{-\frac32}\| \partial \nabla\Gamma^{\le |\alpha|} n^\Delta\|_{2}\|\partial_t\Phi\|_{2}.
\end{split}
\end{equation}
Then by combining \eqref{3.27} and \eqref{3.29}, we obtain that 
\begin{align}
&\frac{d}{dt}(\| \partial \Gamma^{\le K}\phi\|_{2}^2+\| \Gamma^{\le K}\phi\|_{2}^2+\|\partial\nabla\Gamma^{\le K} n^\Delta\|_{2}^2)
\\\le&\  (C\varepsilon)^2t^{-\frac32+\delta_1}\| \partial \Phi\|_{2}^2+(C\varepsilon)^2t^{-\frac32}\| \partial \Phi\|_{2}+ \ C\varepsilon t^{-\frac32}(\|\nabla\Gamma^{\le K}\phi\|_{2}^2+\|\Gamma^{\le K}\phi\|_{2}^2+\| \partial_t\nabla\Gamma^{\le K} n^\Delta\|_{2}^2).
\end{align}
This implies that 
\begin{align}
\| \partial \Gamma^{\le K}\phi\|_{2}+\| \Gamma^{\le K}\phi\|_{2}+\|\partial\nabla\Gamma^{\le K} n^\Delta\|_{2}
\le C\varepsilon.
\end{align}
Here we need the following smallness assumption:
\begin{align}\label{3.35a}
    \| (\partial \Gamma^{\le K}\phi)|_{t=0}\|_{2}+
\| (\Gamma^{\le K}\phi)|_{t=0}\|_{2}+\|(\partial\nabla\Gamma^{\le K} n^\Delta)|_{t=0}\|_{2}\le C_i\varepsilon,    
\end{align}
where the positive constant $C_i$ is clear from the initial condition (cf. \eqref{B1}-\eqref{B3}).

Furthermore it follows from \eqref{3.21} and \eqref{3.13} that
\begin{align}
\|\Gamma^{\le K}n\|_{2}\le C\varepsilon.
\end{align}
Thus \eqref{3.32} holds.

Next we show \eqref{3.34}.
Let $\alpha$ be a multi-index and $|\alpha|=m\le k+4$. Then we have
\begin{align}\label{3.35}
\Box \Gamma^\alpha n^\Delta=\sum_{\alpha_1+\alpha_2= \alpha}C_{\alpha;\alpha_1,\alpha_2}\Gamma^{\alpha_1}\phi\Gamma^{\alpha_2}\phi.
\end{align}
 Multiplying both sides of \eqref{3.35} by $\partial_t\Gamma^\alpha n^\Delta$, we obtain
\begin{align}
\frac{d}{dt}\|\partial\Gamma^\alpha n^\Delta\|_{2}^2
\lesssim& \sum_{\alpha_1+\alpha_2= \alpha}\int|\Gamma^{\alpha_1}\phi\Gamma^{\alpha_2}\phi \partial_t\Gamma^\alpha n^\Delta|
 \\ \lesssim &\|\Gamma^{\le [\frac{m}2]}\phi\|_{\infty}\|\Gamma^{\le m}\phi\|_{L^2}
              \| \partial_t\Gamma^\alpha n^\Delta\|_{2}
\\ \le & \ (C\varepsilon)^2 t^{-\frac32}\| \partial_t\Gamma^\alpha n^\Delta\|_{2} \qquad\qquad (\text{by \eqref{A3} and \eqref{3.32}}),\\ \implies \|\partial\Gamma^\alpha n^\Delta\|_{2}\lesssim& \|(\partial\Gamma^{\alpha}n^{\Delta})|_{t=0}\|_{2}+(C\varepsilon)^2.\label{3.42}
\end{align}
Since interchanging the vector fields merely generates lower-order terms, we estimate the initial data following similar arguments in \cite{CLLX} (see Lemma 2.6 therein): 
\begin{align}
\|(\partial\Gamma^{\alpha}n^{\Delta})|_{t=0}\|_{2}
\lesssim\|(\Gamma_{\Omega}^{\alpha_1}\partial_{t}^{\alpha_2}\nabla^{\alpha_3}n^{\Delta})|_{t=0}\|_{2}
\lesssim \sum_{\alpha_1+\alpha_2+\alpha_3\le K+1}\sum_{J=0}^{\alpha_1}\|\langle x\rangle^{J}(\partial^J\partial_{t}^{\alpha_2}\nabla^{\alpha_3}n^{\Delta})|_{t=0}\|_2,
\end{align}
where $\Gamma_{\Omega}=\Omega_{i0},\Omega_{ij}$. Recalling the equation of $n^\Delta$ in \eqref{3.333}, we have for any $b\ge2$ and $a\ge0$
\[\partial_t^b\nabla^a n^\Delta=\partial_t^{b-2}\nabla^a \Delta n^\Delta+\partial_t^{b-2}\nabla^a(|\phi|^2).
\]
Since $(n^{\Delta},\partial_t n^{\Delta})|_{t=0}=\mathbf 0$, one has (by Proposition \ref{prop})
\begin{align}
\sum_{J=0}^{\alpha_1}\|\langle x\rangle^{J}(\partial^J\partial_{t}^{\alpha_2}\nabla^{\alpha_3}n^{\Delta})|_{t=0}\|_2
\lesssim \|\langle x\rangle^{[\frac{K+1}2]+1}\langle \partial \rangle^{ K-1}\phi|_{t=0}\|_2\|\langle x\rangle^{[\frac{K+1}2]}\langle \partial \rangle^{ [\frac{K+1}{2}]}\phi|_{t=0}\|_\infty\le C\varepsilon^2.\label{3.44}
\end{align}
The cases $b=0,1$ follow easily. Hence \eqref{3.34} holds combining the estimates \eqref{3.42} and \eqref{3.44}. The initial condition of the Klein-Gordon component can be handled similarly and we refer the analysis to Proposition~\ref{prop}.
\begin{rem}\label{rem3.2}
 Note that here even though the initial data $(\Gamma^\alpha n^\Delta,\partial_t \Gamma^\alpha n^\Delta)|_{t=0}$ are not zero, it is of the order $O(\varepsilon^2)$ thanks to the nonlinear structure. Therefore one can treat $n^\Delta$ as a higher order perturbation of $n^0$.
\end{rem}

\end{proof}

\section {optimal pointwise decay in time and scattering}
In this section we complete the proof of Theorem~\ref{thm} by showing the $a\ priori$ hypothesis and the scattering. We first prove the two assumptions:

\subsection{Assumption 1: wave component \eqref{A2}}
Recall $n=n^0+\Delta n^\Delta$. Then we separate the proof into two parts: given multi-index $|\alpha|=k$,
\begin{align}
&|\Gamma^\alpha n^0|\le C\varepsilon \langle r-t\rangle^{-\frac12}\langle r+t\rangle^{-1},\label{4.1}\\
&|\Gamma^\alpha \Delta n^\Delta|\le C\varepsilon \langle r-t\rangle^{-\frac12}\langle r+t\rangle^{-1}.\label{4.2}
\end{align} 
\texttt{Homogeneous part $n^0$ \eqref{4.1}}:  

Since the equation of $\Gamma^{\alpha}n^0$ is not related to $\phi$ then by the Klainerman Sobolev inequality \eqref{ksineqa} we can gain a better decay for the homogeneous part with a careful analysis of the conformal energy:
\begin{align}
|\Gamma^{\alpha} n^0|\le  & C_{ks}\langle r-t\rangle^{-\frac12}\langle r+t\rangle^{-1}\|\widetilde\Gamma^{\le k+3}n^{0}\|_{2}\qquad \text{(by \eqref{ksineqa})}
\\\le&\ C_{ks}\varepsilon\langle r-t\rangle^{-\frac12}\langle r+t\rangle^{-1}\qquad \text{(by \eqref{3.130})},\label{3.70}
\end{align}
{where $C_{ks}$ is the known constant from the Klainerman-Sobolev embedding \eqref{ksineqa}.}

\texttt{Inhomogeneous part $n^\Delta$ \eqref{4.2}:}

Recall from \eqref{3.31} that we have
\begin{align}
|\Gamma^{\alpha}\Delta n^\Delta|\lesssim |\Delta \Gamma^{\le |\alpha|}n^\Delta|+ |\nabla \partial_t \Gamma^{\le |\alpha|}n^\Delta|\lesssim |\partial\nabla \Gamma^{\le |\alpha|} n^{\Delta}|.
\end{align}
Therefore it suffices to show that 
\begin{align}\label{3.3}
|\partial\nabla \Gamma^{\le |\alpha|} n^{\Delta}|
\le (C\varepsilon)^2\langle r-t\rangle^{-\frac12}\langle r+t\rangle^{-1}.
\end{align}
For this purpose, we consider the two cases $r<\frac t2$ and $r\ge \frac t2$. With no loss we assume that $t>2$.

{\bf Case 1: $r<\frac t2$}. 
Recall \eqref{2.260}: $|\langle r-t\rangle\partial \nabla u|\le |\partial \Gamma^{\le 1}u |+(t+r)|\Box u|$ for all $r>0$. We obtain 
\begin{align}
|\langle r-t\rangle \partial\nabla \Gamma^{\alpha} n^{\Delta}|\lesssim &| \partial\Gamma^{\le |\alpha|+1} n^{\Delta}|+(t+r)|\Box  \Gamma^{\alpha}n^{\Delta}|
\\ \lesssim &\ t^{-\frac34}\|\partial \Gamma^{\le |\alpha|+4} n^{\Delta}\|_{2}+\|\langle t+|\cdot|\rangle \Gamma^{\le [\frac {|\alpha|} 2]}\phi\|_{\infty}\|\Gamma^{\le |\alpha|}\phi \|_{\infty}\qquad (\text{by }\eqref{t34})
\\ \lesssim &\ \varepsilon^2 t^{-\frac34}+\varepsilon^2t^{-2}\qquad(\text{by \eqref{3.34} and \eqref{A3}}).\label{3.4}
\end{align}

{\bf Case 2: $r\ge\frac t2$}. By \eqref{t1}, we obtain 
\begin{align}
\langle r-t\rangle^{\frac12} t|\partial \nabla \Gamma^{\alpha} n^{\Delta}|
 \lesssim& \|\partial\Gamma^{\le |\alpha|+3} n^{\Delta}\|_{2}+\|\langle t-|\cdot|\rangle\partial\nabla\Gamma^{\le |\alpha|+3} n^{\Delta}\|_{2}
\\ \lesssim &\|\partial\Gamma^{\le |\alpha|+4}n^{\Delta}\|_{2}+\|\langle t+|\cdot|\rangle\Box\Gamma^{\le |\alpha|+3}n^{\Delta}\|_{2}\qquad \text{(by \eqref{2.260})}
\\ \lesssim &\  \varepsilon^2+\|\langle t+|\cdot|\rangle\Box\Gamma^{\le |\alpha|+3}n^{\Delta}\|_{2}\qquad \text{(by \eqref{3.34})}.
\end{align}
Recall $\Box n^{\Delta}=|\phi|^2$. Let $\beta$ be a multi-index and $|\beta|\le |\alpha|+3$, we have 
\[\Box\Gamma^{\beta}n^{\Delta}=\sum_{\beta_1+\beta_2= \beta}C_{\beta;\beta_1,\beta_2}\Gamma^{\beta_1}\phi\Gamma^{\beta_2}\phi.\]
This implies that
\begin{align}
\|\langle t+|\cdot|\rangle \Box\Gamma^{\le |\alpha|+3}n^{\Delta}\|_{2}\lesssim \|\Gamma^{\le |\alpha|+3}\phi\|_{2}\|\langle t+|\cdot|\rangle \Gamma^{\le [\frac{|\alpha|+3}2]}\phi\|_{\infty}
\le (C\varepsilon)^2t^{-\frac12} 
\qquad \text{(by \eqref{3.32} and \eqref{A3})}.
\end{align}
It follows that 
\begin{align}
\langle r-t\rangle^{\frac12} t|\partial \nabla \Gamma^{\alpha} n^{\Delta}|
 \le (C\varepsilon)^2,\qquad \text{for }r\ge\frac t2.
\end{align}
Thus we have proved \eqref{3.3}. Then it follows that
\begin{align}\label{3.33}
|\Gamma^{\alpha}n|\le& |\Gamma^{\alpha}n^0|+|\Gamma^{\alpha}\Delta n^{\Delta}|
\\ \le &\ C_{ks}\varepsilon\langle r-t\rangle^{-\frac12}\langle r+t\rangle^{-1}+(C\varepsilon)^2\langle r-t\rangle^{-\frac12}\langle r+t\rangle^{-1}
\qquad \text{(by \eqref{3.70} and \eqref{3.3})}
\\ \le&\ (C_{ks}\varepsilon+\widetilde {C}_{n}\varepsilon^2)\langle r-t\rangle^{-\frac12}\langle r+t\rangle^{-1}.
\end{align}

{ Then it suffices to choose $C_1>C_{ks}+\widetilde {C}_{n}\varepsilon$ in \eqref{A2}. }

\subsection{Assumption 2: Klein-Gordon component \eqref{A3}}

Let $\beta$ be a multi-index and $|\beta|\le k$. Then  $\Gamma^{\beta}\phi$ satisfies
\begin{equation}
\begin{cases}
\Box \Gamma^{\beta}\phi +\Gamma^{\beta}\phi=-\sum_{\beta_1+\beta_2=\beta}C_{\beta;\beta_1,\beta_2}\Gamma^{\beta_1}n\Gamma^{\beta_2}\phi=:\widetilde F,\\
(\phi,\partial_t\phi)|_{t=0}=(\phi_0,\phi_1).
\end{cases}
\end{equation}
By \eqref{d-1}, we have 
\begin{align}
\langle t+|x|\rangle^{\frac32}|\Gamma^{\beta}\phi(t,x)|\lesssim{\sup_{0<s\le t}}\|\langle s+|\cdot|\rangle^{1+\delta}\Gamma^{\le 4}\widetilde F(s,\cdot)\|_{L^2}+\|\langle |\cdot|\rangle^{2} \Gamma^{\le 4}\Gamma^{\beta}\phi(0,\cdot)\|_{L^2},
\end{align}
where $0<\delta \ll 1$.
For any $s\in(0,t)$ we have
\begin{align}
&\|{\langle s+|\cdot|\rangle^{1+\delta}}\Gamma^{\le 4}\widetilde F(s,\cdot)\|_{L^2}
\\\lesssim& \|\Gamma^{\le |\beta|+4}n\|_{L^2}\|\langle s+|\cdot|\rangle^{1+\delta}\Gamma^{\le [\frac{|\beta|+4}2]}\phi\|_{L^\infty}
+\|\langle s+|\cdot|\rangle^{1+\delta}\Gamma^{\le [\frac{|\beta|+4}2]}n\Gamma^{\le |\beta|+4}\phi\|_{L^2}
\\\lesssim &\ \varepsilon^2 +\|\langle s+|\cdot|\rangle^{1+\delta}\Gamma^{\le [\frac{|\beta|+4}2]}n\Gamma^{\le |\beta|+4}\phi\|_{L^2}
 \qquad \text{(by \eqref{A3} and \eqref{3.32})}.\label{3.59}
\end{align}
For the last term in \eqref{3.59},
we recall \eqref{kgw}: {for 
$u$ satisfying} $\Box u+u=F$, one has
\begin{align}
\Big|\frac{\langle t+r\rangle}{\langle t-r\rangle}u\Big|
\lesssim&|\partial\Gamma^{\le 1}u|+\Big|\frac{\langle t+r\rangle}{\langle t-r\rangle}F\Big|.
\end{align}
 Let $u=\Gamma^{\mu}\phi$ with $|\mu|\le |\beta|+4$. 
Since $\Box \Gamma^{\mu}\phi+\Gamma^{\mu}\phi=\sum_{\mu_{1}+\mu_2= \mu}C_{\mu;\mu_1,\mu_2}\Gamma^{\mu_1}n\Gamma^{\mu_2}\phi$, it follows that 
\begin{align}
 \|\frac{\langle t+|\cdot|\rangle}{\langle t-|\cdot|\rangle}\Gamma^{\mu}\phi\|_{2}
\lesssim&\|\partial\Gamma^{\le |\mu|+1}\phi\|_{2}+\sum_{\mu_{1}+\mu_2= \mu}\|\frac{\langle t+|\cdot|\rangle}{\langle t-|\cdot|\rangle}\Gamma^{\mu_1}n\Gamma^{\mu_2}\phi\|_2
 \lesssim   \varepsilon. \label{3.680}
\end{align}
Thus
\begin{align}
\|\langle s+|\cdot|\rangle^{1+\delta}\Gamma^{\le [\frac{|\beta|+4}2]}n\Gamma^{\le |\beta|+4}\phi\|_{2}
\lesssim& \|\langle s+|\cdot|\rangle^{\frac34}\langle s-|\cdot|\rangle^{\frac12}\Gamma^{\le [\frac{|\beta|+4}2]}n\|_{L^\infty}\|\frac{\langle s+|\cdot|\rangle^{\frac14+\delta}}{\langle s-|\cdot|\rangle^{\frac12}}\Gamma^{\le |\beta|+4}\phi\|_{2}
\\ \le&\ C\varepsilon \|\frac{\langle s+|\cdot|\rangle}{\langle s-|\cdot|\rangle}\Gamma^{\le |\beta|+4}\phi\|_{2} ,\qquad(\text{by \eqref{A2}})
\\ \le&\ C\varepsilon^2,\qquad(\text{by  \eqref{3.680}}).
\end{align}
Here the constant $C$ may vary from line to line. It then follows that 
\begin{align}
\langle t+r\rangle^{-\frac32}|\Gamma^{\beta}\phi(t,x)|\le& C\varepsilon^2+
\|\langle |\cdot|\rangle^2\Gamma^{\le 4}\Gamma^{\beta}\phi(0,\cdot)\|_{2}
\\ \le& C\varepsilon^2+C_i\varepsilon\qquad\text{(by \eqref{ID2} and \eqref{B3})}
\\ \le& C_i\varepsilon+\widetilde C_{\phi}\varepsilon^2.
\end{align}
Then it suffices to choose $C_2>{C_i}+\widetilde {C}_{\phi}\varepsilon$ in \eqref{A3},
{where $C_i$ is the known constant from the initial condition (cf. \eqref{B3}).}

\subsection{Scattering}\label{scat}
It remains to show the scattering of $(\phi,n)$ in the energy space $\mathcal{X}_K=H^{K+1}\times H^{K}\times H^K\times H^{K-1} $, namely there exists a couple of solutions to the linear system $(\phi_l, n_l)$ such that  
\begin{align}
    \lim_{t\to \infty}\|(\phi,\partial_t\phi, n,\partial_t n)-(\phi_l,\partial_t\phi_l, n_l,\partial_t n_l)\|_{\mathcal{X}_K}= 0,
\end{align}
where $(\phi,\partial_t\phi, n,\partial_t n)|_{t=0}=(\phi_0,\phi_1,n_0,n_1)$ and $(\phi_l,\partial_t\phi_l, n_l,\partial_t n_l)|_{t=0}=(\phi_{l_0},\phi_{l_1},n_{l_0},n_{l_1}) $ for some $(\phi_{l_0},\phi_{l_1},n_{l_0},n_{l_1})$.
Recall that for the Cauchy problem for the following inhomogeneous 3d wave system:
\begin{align}
    \Box n= G\, ,\ \mbox{with } (n,\partial_t n)|_{t=0}=(n_0,n_1),
\end{align}
the solution $\vec{n}=(n,\partial_t n)$ is given by
\begin{align}
    \vec{n}(t,x)=M(t)\begin{pmatrix}
n_0\\
n_1
\end{pmatrix}+\int_0^t M(t-s)\begin{pmatrix}
0\\
G(s)
\end{pmatrix}\ ds,
\end{align}
where 
\begin{align}
    M(t)=\begin{pmatrix}
\cos(t |\na|) & |\na|^{-1}\sin(t|\na|)\\
-|\na|\sin(t|\na|) & \cos(t|\na|)
\end{pmatrix}.
\end{align}
Similarly, for the Cauchy problem for the inhomogeneous 3d Klein-Gordon system:
\begin{align}
    \Box \phi+\phi = H\, ,\ \mbox{with } (\phi,\partial_t \phi)|_{t=0}=(\phi_0,\phi_1),
\end{align}
the solution $\vec{\phi}=(\phi,\partial_t \phi)$ is given by
\begin{align}
    \vec{\phi}(t,x)=N(t)\begin{pmatrix}
\phi_0\\
\phi_1
\end{pmatrix}+\int_0^t N(t-s)\begin{pmatrix}
0\\
H(s)
\end{pmatrix}\ ds,
\end{align}
where 
\begin{align}
    N(t)=\begin{pmatrix}
\cos(t \Lg\na\Rg) & \Lg \na\Rg^{-1}\sin(t\Lg\na\Rg)\\
-\Lg\na\Rg\sin(t\Lg\na\Rg) & \cos(t\Lg\na\Rg)
\end{pmatrix}.
\end{align}
More specifically $G(\phi)=\Delta |\phi|^2$ and $H(n,\phi)=-n\phi$ in our setting.
 By a standard semi-group argument, it is known that the wave flow $M(t)$ and Klein-Gordon flow $N(t)$ are unitary semi-groups on $H^{J}\times H^{J-1}$ for $J\in \mathbb{N}^+$. More precisely one can verify on the Fourier side that for any $J\in \mathbb{N}^+$ one has
 \begin{align}
\|M(t)\|_{\mathcal{H}^J\to\mathcal{H}^J }=\|N(t)\|_{\mathcal{H}^J\to\mathcal{H}^J }=1,
 \end{align}
where $\mathcal{H}^J=H^J\times H^{J-1}$. Moreover, one recalls that 
\begin{align}
\begin{pmatrix}
        n_l\\
        \partial_t n_l
    \end{pmatrix}=M(t)
    \begin{pmatrix}
        n_{l_0}\\
        n_{l_1}
    \end{pmatrix}
   \ ,\  \begin{pmatrix}
        \phi_l\\
        \partial_t\phi_l
    \end{pmatrix}=N(t)
    \begin{pmatrix}
        \phi_{l_0}\\
        \phi_{l_1}
    \end{pmatrix}.
\end{align}
Note that by \eqref{3.32}-\eqref{3.34}, we have 
\begin{align}
    \int_0^\infty \| G(\phi)\|_{H^{K-1}}+\|H(\phi,n)\|_{H^{K}}\le C\varepsilon^2.
\end{align}
In particular, we set
\begin{align}
    \begin{pmatrix}
        n_{l_0}\\n_{l_1}
    \end{pmatrix}
    =\begin{pmatrix}
        n_{0}\\n_{1}
    \end{pmatrix}+\int_0^\infty M(-s)
    \begin{pmatrix}
        0\\
        G(\phi)
    \end{pmatrix}\ ds,
\end{align}
and 
\begin{align}
    \begin{pmatrix}
        \phi_{l_0}\\\phi_{l_1}
    \end{pmatrix}
    =\begin{pmatrix}
        \phi_{0}\\\phi_{1}
\end{pmatrix}+\int_0^\infty N(-s)
    \begin{pmatrix}
        0\\
        H(\phi,n)
    \end{pmatrix}\ ds,
\end{align}
then since $M(t)$ and $N(t)$ are unitary semi-groups we have 
\begin{align}
    \|\begin{pmatrix}
        \phi\\ \partial_t \phi    \end{pmatrix}(t)-N(t)
        \begin{pmatrix}
        \phi_{l_0}\\\phi_{l_1}    
        \end{pmatrix}
        \|_{\mathcal{H}^{K+1}}\le \int_t^\infty \| H(\phi,n)\|_{H^K}\ ds,
\end{align}
and
\begin{align}
    \|\begin{pmatrix}
        n\\ \partial_t n    \end{pmatrix}(t)-M(t)
        \begin{pmatrix}
        n_{l_0}\\n_{l_1}   
        \end{pmatrix}        \|_{\mathcal{H}^{K}}\le \int_t^\infty \| G(\phi)\|_{H^{K-1}}\ ds.
\end{align}
Therefore we have shown that $(\phi,n)$ scatters to $(\phi_l,n_l)$ as $t\to+\infty$:
\begin{align}
    \lim_{t\to \infty}\|(\phi,\partial_t\phi, n,\partial_t n)-(\phi_l,\partial_t\phi_l, n_l,\partial_t n_l)\|_{\mathcal{X}_K}= 0.
\end{align}

\appendix

\section{An alternative approach to to estimate $\|\Gamma^{\beta}n^0\|_{L^2}^2$.}

In this section we give an alternative approach to estimate $\|\Gamma^{\beta}n^0\|_{L^2}^2$ as in Lemma~\ref{lem3.1}. It is clear that $\Box \Ga^\beta n^0=0.$ Therefore it is known that we can write $\Ga^\beta n^0$ in the following mild form: 
\begin{align}
    \Ga^\beta n^0(x)= (\cos tD\  v_0)(x)+(\frac{\sin tD}{D}v_1)(x),
\end{align}
where $D=|\na|$ and $(v_0,v_1)=(\Ga^\beta n^0,\partial_t\Ga^\beta n^0)|_{t=0}$.
Then it follows that $\| \Ga^{\beta} n^0\|_2\le \|\cos tD\ v_0\|_2+\|\frac{\sin tD}{D}v_1\|_2$.
Similar to the 2D case as in \cite{CLLX}, one easily show from the Fourier side that
\begin{align}
    \|\cos tD\ v_0\|_{L^2_x(\R^3)}\sim \|\cos(t|\xi|) \hat{v_0}(\xi)\|_{L^2_\xi(\R^3)}\lesssim\|\hat{v_0}\|_{L^2_\xi(\R^3)}\lesssim \|v_0\|_{L^2_x(\R^3)}. 
\end{align}
Similarly by the Hardy's inequality,
\begin{align}
 \|\frac{\sin tD}{D}v_1\|_{L^2_x(\R^3)}^2\lesssim\int_{\R^3} \frac{(\sin(t|\xi|)^2}{|\xi|^2} |\hat{v_1}|^2\ d\xi\lesssim \int_{\R^3} \frac{|\hat{v_1}|^2}{|\xi|^2} \ d\xi\lesssim \|\hat{v_1}\|^2_{H^1_\xi(\R^3)}\lesssim\|\Lg x\Rg v_1\|^2_{L^2_x(\R^3)}.
\end{align}
\begin{rem}
   Note that compared to the $X(\partial)$ trick (conformal energy estimates), this propagator estimate needs an additional assumption:
\begin{align}\label{4.48}
    \|v_0\|_2+\|\Lg x\Rg v_1\|_2\lesssim\sum_{j=0}^K\|\Lg x\Rg^{j+1}\na^j(\na n_0,n_1)\|_2\le C\varepsilon.
\end{align}
One clearly sees the difference between \eqref{3.12a} and \eqref{4.48}.
\end{rem}

\section*{Acknowledgement}
X. Cheng was partially supported by the Shanghai ``Super Postdoc" Incentive Plan and China Postdoctoral Science Foundation (Grant No. 2022M710796, 2022T150139).

\bibliographystyle{abbrv}

\end{document}